\documentclass[11pt,reqno]{amsart}
\usepackage{amsmath, amsfonts, color, amsthm, graphics, amssymb,lscape,ulem}
\usepackage[notcite,notref,final]{showkeys}

\date{}

\newcommand{\ra}{\rightarrow}
\newcommand{\lra}{\longrightarrow}

\newcommand{\mc}{\mathcal}

\theoremstyle{plain}
\newtheorem{theorem}{Theorem}
\newtheorem{prop}[theorem]{Proposition}
\newtheorem{corollary}[theorem]{Corollary}
\newtheorem{lemma}[theorem]{Lemma}

\theoremstyle{definition}

\newtheorem{remark}[theorem]{Remark}

\numberwithin{equation}{section}

\input xy
\xyoption{all}

\begin{document}

\title[Corrigendum: Degenerate Sklyanin algebras]{Corrigendum to ``Degenerate Sklyanin algebras and Generalized Twisted Homogeneous Coordinate rings",\\ J. Algebra 322 (2009) 2508-2527}
\author{Chelsea Walton}

\address{Department of Mathematics, University of Washington, Seattle, WA 98103.}
\email{notlaw@math.washington.edu}
\date{}


\bibliographystyle{abbrv}

\maketitle


\begin{abstract}
There is an error in the computation of the truncated point schemes $V_d$ of the degenerate Sklyanin algebra $S(1,1,1)$. We are grateful to S. Paul Smith for pointing out that $V_d$ is larger than was claimed in Proposition 3.13. All 2 or 3 digit references are to the above paper, while 1 digit references are to the results in this corrigendum. We provide a description of the correct $V_d$ in Proposition \ref{prop:Vd} below. Results about the corresponding point parameter ring $B$ associated to the schemes $\{V_d\}_{d \geq 1}$ are given afterward. 
\end{abstract}

\section{Corrections}

The main error in the above paper is to the statement of Lemma 3.10. Before stating the correct version, we need some notation.

\medskip

\noindent {\it Notation.} Given $\zeta = e^{2 \pi i /3}$, let $p_a:=[1:1:1]$, $p_b:=[1:\zeta:\zeta^2]$, and $p_c:=[1:\zeta^2:\zeta]$. Also, let $\check{\mathbb{P}}_A^1 := \mathbb{P}_A^1 \setminus \{p_b, p_c\}$, $\check{\mathbb{P}}_B^1 := \mathbb{P}_B^1 \setminus \{p_a, p_c\}$, and
$\check{\mathbb{P}}_C^1 := \mathbb{P}_C^1 \setminus \{p_a,p_b\}$.

\medskip

We also require the following more precise version of Lemma 3.9; the original result is correct though there is a slight change in the proof as given below.

\begin{lemma}\label{lem:Lemma3.9} (Correction of Lemma 3.9) 
Let $p=(p_0, \dots, p_{d-2}) \in V_{d-1}$ with $p_{d-2} \in \check{\mathbb{P}}_A^1$, $\check{\mathbb{P}}_B^1,$ or $\check{\mathbb{P}}_C^1$. If $p'=(p,p_{d-1}) \in V_d$, then $p_{d-1}$ = $p_a$, $p_b$,  or $p_c$ respectively. 
\end{lemma}

\begin{proof}
The proof follows from that of Lemma 3.9, except that there is a typographical error in the case when $p_{d-2} = [0: y_{d-2} : z_{d-2}]$. Here, we require that $(p_{d-2},p_{d-1})$ satisfies the system of equations:
\begin{center}
$\begin{array}{r}
f_{d-2}=g_{d-2}=h_{d-2}=0, \\
y_{d-2}^3 + z_{d-2}^3 = 0,\\
x_{d-1}^3 + y_{d-1}^3 + z_{d-1}^3 - 3 x_{d-1} y_{d-1} z_{d-1} = 0.
\end{array}$
\end{center}
This implies that either $y_{d-2} = z_{d-2} = 0$ or $x_{d-1} = y_{d-1} = z_{d-1} = 0$, which produces a contradiction.
\end{proof}

Now the correct version of Lemma 3.10 is provided below. The present version is slightly weaker than the original result, where it was claimed that $p_{d-1} \in \check{\mathbb{P}_{*}^1}$ instead of $p_{d-1} \in \mathbb{P}_{*}^1$. Here, $\mathbb{P}_{*}^1$ denotes either $\mathbb{P}_{A}^1$, $\mathbb{P}_{B}^1$, or $\mathbb{P}_{C}^1$.

\begin{lemma} \label{lem:Lemma3.10} (Correction of Lemma 3.10) 
Let $p=(p_0, \dots, p_{d-2}) \in V_{d-1}$ with $p_{d-2}$ = $p_a$, $p_b$,  or $p_c$.  If $p'=(p,p_{d-1}) \in V_d$, then $p_{d-1} \in \mathbb{P}_A^1$, $\mathbb{P}_B^1,$ or $\mathbb{P}_C^1$ respectively. 
\end{lemma}

\begin{proof} The proof follows from that of Lemma 3.10 with the exception that there is a typographical error in the definition of the function $\theta$; it should be defined as:
\begin{center}
$\theta(y_{d-1}, z_{d-1}) = \left\{\begin{array}{ll}
                      -(y_{d-1} + z_{d-1}) & \text{if~} p_{d-2} = p_a,\\
                      -(\zeta^2 y_{d-1} + \zeta z_{d-1}) & \text{if~} p_{d-2} = p_b,\\
                      -(\zeta y_{d-1} + \zeta^2 z_{d-1}) & \text{if~} p_{d-2} = p_c.
                      \end{array} \right.$
                      \end{center}
\vspace{-.2in}

\end{proof}

\begin{remark} There are two further minor typographical corrections to the paper.
\begin{enumerate}
\item (Correction of Figure 3.1)
The definition of the projective lines $\mathbb{P}_B^1$ and $\mathbb{P}_C^1$ should be interchanged. More precisely, the curve $E_{111}$ is the union of three projective lines: 
{\small
\[
\begin{array}{l}
\mathbb{P}^1_A: x+y+z=0,\\
\mathbb{P}^1_B: x+\zeta^2 y+\zeta z=0,\\
\mathbb{P}^1_C: x+\zeta y + \zeta^2 z=0.
\end{array}\]
}
\item (Correction to Corollary 4.10) 
The numbers 57 and 63 should be replaced by 24 and 18 respectively.
\end{enumerate}
\end{remark}

\section{Consequences} 

The main consequence of weakening Lemma 3.10 to  Lemma 3  is that the truncated point schemes $\{V_d\}_{d \geq 1}$ of $S=S(1,1,1)$  are strictly larger than the truncated point schemes computed in Proposition 3.13 for $d \geq 4$. We discuss such results in $\S$2.1 below. Furthermore, the corresponding point parameter ring associated to the correct point scheme data of $S$ is studied in $\S$2.2.

\medskip

\noindent {\it Notation.}  (i) Let $W_d := \bigcup_{i=1}^6 W_{d,i}$ with $W_{d,i}$ defined in Proposition 3.13. 

\noindent (ii) Let $B:= \bigoplus_{d \geq 0} H^0(V_d, \mc{O}_{V_d}({\bf 1}))$ be the point parameter ring of $S(1,1,1)$ as in Definition 1.8. 

\noindent (iii) Likewise let $P:= \bigoplus_{d \geq 0} H^0(W_d, \mc{O}_{W_d}({\bf 1}))$ be the point parameter ring associated to the schemes $\{W_d\}_{d \geq 1}$.

\medskip

The results of $\S$4 of the paper are still correct; we describe the ring $P$, and we show that it is a factor of $S(1,1,1)$. Unfortunately, the ring $P$ is not equal to the point parameter ring $B$ of $S(1,1,1)$. More precisely, the following corrections should be made.

\begin{remark} \label{rmk:newVd} 
(1) The scheme $V_d$ should be replaced by $W_d$ in Theorem 1.7, in Proposition 3.13, in Remark 3.14, and in all $\S$4 after Definition 4.1. 
\smallskip

\noindent (2) The ring $B$ should be replaced by $P$ in $\S$1 after Definition 1.8, and in all $\S$4 with the exception of the second paragraph.
\end{remark}

\subsection{On the truncated point schemes $\{V_d\}_{d \geq 1}$}
We provide a description of the truncated point schemes $\{V_d\}_{d \geq 1}$ as follows. 

\medskip

\noindent {\it Notation.} Let $\{V_{d,i}\}_{i \in I_d}$ denote the $|I_d|$ irreducible components of the $d^{\text{th}}$ truncated point scheme $V_d$.

\begin{prop} \label{prop:Vd} (Description of $V_d$)
For $d \geq 2$, the  length $d$ truncated point scheme $V_d$ is realized as the union of length $d$ paths of the quiver $Q$ below. 
With $d=2$, for example, the path $\mathbb{P}^1_A \lra p_a$ corresponds to the component $\mathbb{P}^1_A \times p_a$ of $V_2$. 
\end{prop}

{\small
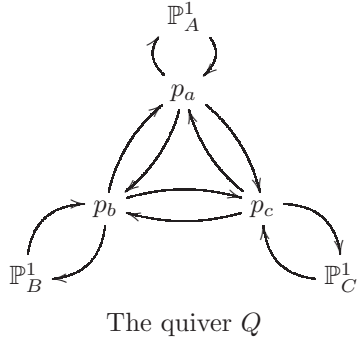
\[
\xymatrix@-1pc{
&&                 \mathbb{P}_A^1 \ar@/^1pc/[d]   &&\\
&&                           p_a \ar@/^1pc/[u]     \ar@/^/[rdd] \ar@/^/[ldd]              &&\\\\
& p_b   \ar@/^1pc/[ld]      \ar@/^/[ruu] \ar@/^/[rr]               & &                 p_c  \ar@/^1pc/[rd]  \ar@/^/[luu] \ar@/^/[ll]&\\
\mathbb{P}_B^1 \ar@/^1pc/[ru]   &&&&   \mathbb{P}_C^1 \ar@/^1pc/[lu]
}
\] 
\begin{center}\mbox{\text{The quiver $Q$}}\end{center}
}

\medskip

\begin{proof}
We proceed by induction.  Considering the $d=2$ case, Lemma 3.12 still holds so $V_2 = W_2$, the union of the irreducible components:

{\small
\[
\begin{array}{lllll}
\mathbb{P}_A^1 \times p_a, && \mathbb{P}_B^1 \times p_b, && \mathbb{P}_C^1 \times p_c\\
p_a \times \mathbb{P}_A^1, && p_b \times \mathbb{P}_B^1, &&p_c \times \mathbb{P}_C^1.\\
\end{array}
\]
}

\noindent One can see these components correspond to length 2 paths of the quiver Q. Conversely, any length 2 path of $Q$  corresponds to a component that lies in $V_2$. 

We assume the proposition holds for $V_{d-1}$, and
recall that Lemmas 2 and 3 provide the recipe to build $V_d$ from $V_{d-1}$. Take a point $(p_0, \dots, p_{d-2}) \in V_{d-1,i}$, where the irreducible component $V_{d-1,i}$ of $V_{d-1}$ corresponds to a length $d-1$ path of $Q$. Let $\{V_{d,ij}\}_{j \in J}$ be the set of $|J|$ irreducible components of $V_d$ with $$(p_0, \dots, p_{d-2}, p_{d-1}) \in V_{d,ij} \subseteq V_d$$ for some  $p_{d-1} \in \mathbb{P}^2$.
There are two cases to consider. 

\smallskip

\noindent \underline{Case 1}: We have that $(p_{d-3},p_{d-2})$ lies in  one of the following products:
\[
\begin{array}{lcclccl}
\mathbb{P}_A^1 \times p_a, &&&\mathbb{P}_B^1 \times p_b,  &&&\mathbb{P}_C^1 \times p_c ,\\
p_a \times \check{\mathbb{P}}_A^1, &&&p_b \times \check{\mathbb{P}}_B^1, &&& p_c \times \check{\mathbb{P}}_C^1.
\end{array}
\]
For the first three choices,  Lemma \ref{lem:Lemma3.10} implies that $pr_d(V_{d,ij})$ = $\mathbb{P}_A^1$, $\mathbb{P}_B^1$, or $\mathbb{P}_C^1$,  respectively. For the second three choices, $p_{d-2}$ belongs to $\check{\mathbb{P}}_A^1$, $\check{\mathbb{P}}_B^1$, or $\check{\mathbb{P}}_C^1$, and Lemma \ref{lem:Lemma3.9} implies that $pr_d(V_{d,ij}) = p_a$, $p_b$, or $p_c$,  respectively. We conclude by induction that  the component $V_{d,ij}$ yields a length $d$ path of $Q$. 

\smallskip

\noindent \underline{Case 2}: We have that $(p_{d-3},p_{d-2})$ is equal to one of the following points:
\[
\begin{array}{cccc}
p_a \times p_b, &&& p_a \times p_c,\\
p_b \times p_a, &&& p_b \times p_c,\\
p_c \times p_a, &&& p_c \times p_b.
\end{array}
\]
Now  Lemma \ref{lem:Lemma3.10} implies that: 
\[
pr_d(V_{d,ij}) = 
\begin{cases}
\mathbb{P}_A^1  & \text{~if~} p_{d-2}  = p_a,\\
\mathbb{P}_B^1  & \text{~if~}p_{d-2}  =  p_b,\\
\mathbb{P}_C^1  & \text{~if~} p_{d-2}  = p_c.\\
\end{cases}
\]
Again we have that in this case, the component $V_{d,ij}$ yields a length $d$ path of $Q$.

\medskip

Conversely (in either case), let $\mc{P}$ be a length $d$ path of $Q$. Then, by induction, the embedded length $d-1$ path $\mc{P}'$ ending at the $d-1^{\text{st}}$ vertex $v'$ of $\mc{P}$ yields a component $X'$ of $V_{d-1}$. Say $v$ is the $d^{\text{th}}$ vertex of $\mc{P}$. If $v'$ is equal to $\mathbb{P}_A^1$, $\mathbb{P}_B^1$, or $\mathbb{P}_C^1$, then $v$ must be $p_a$, $p_b$, or $p_c$ by the definition of $Q$, respectively. Lemma 2 then ensures that $\mc{P}$ yields a component $X$ of $V_d$ so that $pr_{1\dots d-1}(X) = X'$. On the other hand, if $v'$ is equal to $p_a$, $p_b$, or $p_c$, then $v$ lies in $\mathbb{P}_A^1$, $\mathbb{P}_B^1$, or $\mathbb{P}_C^1$, respectively. Likewise, Lemma 3 implies that $\mc{P}$ yields a component $X$ of $V_d$ so that $pr_{1\dots d-1}(X) = X'$.
\end{proof}

\begin{corollary} \label{cor:VdWd} We have that $V_d = W_d$ for $d=1,2,3$, and that $V_d \supsetneq W_d$ for $d \geq 4$. 
\end{corollary}

\begin{proof}
First, $V_1 = \mathbb{P}^2 = W_1$. Next, as mentioned in the proof of Proposition \ref{prop:Vd}, $V_2 = W_2$ is the union of the irreducible components:

{\small
\[
\begin{array}{lllll}
\mathbb{P}_A^1 \times p_a, && \mathbb{P}_B^1 \times p_b, && \mathbb{P}_C^1 \times p_c\\
p_a \times \mathbb{P}_A^1, && p_b \times \mathbb{P}_B^1, &&p_c \times \mathbb{P}_C^1.\\
\end{array}
\]
}

\noindent By Proposition \ref{prop:Vd}, we have that $V_3 = X_{3,1} \cup X_{3,2}$ where $X_{3,1}$ consists of the irreducible components:
{\small
\[
\begin{array}{lllll}
\mathbb{P}_A^1 \times p_a \times \mathbb{P}_A^1, && \mathbb{P}_B^1 \times p_b \times \mathbb{P}_B^1, && \mathbb{P}_C^1 \times p_c \times \mathbb{P}_C^1,\\
 p_a \times \mathbb{P}_A^1 \times p_a, &&p_b \times \mathbb{P}_B^1 \times p_b,  &&p_c \times \mathbb{P}_C^1 \times p_c, 
\end{array}
\]
}

\noindent and $X_{3,2}$ is the union of:

{\small
\[
\begin{array}{lllllll}
\mathbb{P}^1_A \times p_a \times p_b, && \mathbb{P}^1_A \times p_a \times p_c, && p_a \times p_b  \times \mathbb{P}_B^1, && p_a \times p_c  \times \mathbb{P}_C^1,\\
p_a \times p_b \times p_a, && p_a \times p_b \times p_c, && p_a \times p_c \times p_a, && p_a \times p_c \times p_b,\\
\mathbb{P}^1_B \times p_b \times p_c, && \mathbb{P}^1_B \times p_b \times p_a, && p_b \times p_c  \times \mathbb{P}_C^1, && p_b \times p_a  \times \mathbb{P}_A^1,\\
p_b \times p_c \times p_b, && p_b \times p_c \times p_a, && p_b \times p_a \times p_b, && p_b \times p_a \times p_c,\\
\mathbb{P}^1_C \times p_c \times p_a, && \mathbb{P}^1_C \times p_c \times p_b, &&p_c \times p_a  \times \mathbb{P}_A^1, && p_c \times p_b  \times \mathbb{P}_B^1,\\
p_c \times p_a \times p_c, && p_c \times p_a \times p_b, && p_c \times p_b \times p_c, && p_c \times p_b \times p_a.\\
\end{array}
\]
}

\noindent Note that $X_{3,2}$ is contained in $X_{3,1}$; hence $V_3 =  X_{3,1} = W_3$.  Furthermore, one sees that $W_d \subsetneq V_d$ for $d \geq 4$ as follows. The components of $W_d$ are read off the subquiver $Q'$  of $Q$ below.

{\small
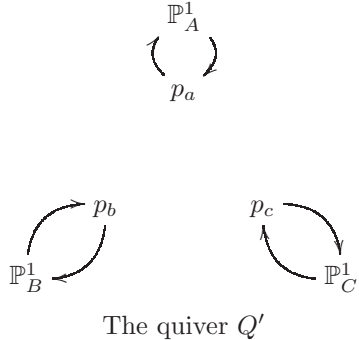
\[
\xymatrix@-1pc{
&&                 \mathbb{P}_A^1 \ar@/^1pc/[d]   &&\\
&&                           p_a \ar@/^1pc/[u]                    &&\\\\
& p_b   \ar@/^1pc/[ld]                  & &                 p_c  \ar@/^1pc/[rd] &\\
\mathbb{P}_B^1 \ar@/^1pc/[ru]   &&&&   \mathbb{P}_C^1 \ar@/^1pc/[lu]
}
\] 
$$\text{The quiver $Q'$}$$
}

On the other hand, for $d \geq 4$, the length $d$ path containing 
$$\mathbb{P}^1_A \lra p_a \lra  p_b \lra \mathbb{P}^1_B$$
corresponds to a component of $V_d$ not contained in $W_d$.
\end{proof}

\subsection{On the point parameter ring $B(\{V_d\})$}

The result that there exists a ring surjection from $S=S(1,1,1)$ onto the ring $P(\{W_d\})$ remains true. However, by Lemma \ref{lem:Bd} below, $B$ is a larger ring than $P$,  and whether there is  a ring surjection from $S$ onto $B$ is unknown. We know that there is a ring homomorphism from $S$ to $B$ with $S_1 \cong B_1$ by \cite[Proposition 3.20]{ATV1}, and computational evidence suggests that $S \cong B$. The details are given as follows.

\begin{lemma} \label{lem:Bd}
The $k$-vector space dimension of $B_d$ is equal to $\dim_k S(1,1,1)_d$ for $d=0,1,\dots,4$. In particular, $\dim_k B_4 \neq \dim_k P_4$.
\end{lemma}

It is believed that analogous computations will show that $\dim_k B_d =  \dim_k S(1,1,1)_d = 3 \cdot 2^{d-1}$ for $d = 5,6$.

\medskip

\noindent {\it Proof of Lemma \ref{lem:Bd}}.
By Corollary \ref{cor:VdWd}, we know that $V_d = W_d$ for $d=1,2,3$;  hence
$$\dim_k B_d = 3 \cdot 2^{d-1} = \dim_k S(1,1,1)_d \text{~~for~} d=0,1,2,3.$$
To compute $\dim_k B_4$, note that by Proposition \ref{prop:Vd}, $V_4$  equals the union $X_{4,1} \cup X_{4,2} \subseteq (\mathbb{P}^2)^{\times 4}$ as follows. Here, $X_{4,1}$ consists of the following irreducible components

{\small
\[
\begin{array}{lll}
\mathbb{P}_A^1 \times p_a \times \mathbb{P}_A^1 \times p_a, &&p_a \times \mathbb{P}_A^1 \times p_a \times \mathbb{P}_A^1,\\
\mathbb{P}_B^1 \times p_b \times \mathbb{P}_B^1 \times p_b, &&p_b \times \mathbb{P}_B^1 \times p_b  \times \mathbb{P}_B^1,\\
\mathbb{P}_C^1 \times p_c \times \mathbb{P}_C^1 \times p_c, &&p_c \times \mathbb{P}_C^1 \times p_c \times \mathbb{P}_C^1;
\end{array}
\]
}

\noindent and $X_{4,2}$ is the union of
{\small
\[
\begin{array}{lll}
\mathbb{P}_A^1 \times p_a \times p_b  \times \mathbb{P}_B^1, &&\mathbb{P}_A^1 \times p_a \times p_c  \times \mathbb{P}_C^1,\\
\mathbb{P}_B^1 \times p_b \times p_a  \times \mathbb{P}_A^1, &&\mathbb{P}_B^1 \times p_b \times p_c  \times \mathbb{P}_C^1,\\
\mathbb{P}_C^1 \times p_c \times p_a  \times \mathbb{P}_A^1, &&\mathbb{P}_C^1 \times p_c \times p_b  \times \mathbb{P}_B^1.
\end{array}
\]
}

\noindent We consider a component such as $\mathbb{P}^1_A \times p_a \times p_b \times p_a$ contained in 
$\mathbb{P}^1_A \times p_a \times p_b \times \mathbb{P}^1_B$ to be included as part of $X_{4,2}$.

Since $X_{4,1} = W_4$ we get that $h^0(\mc{O}_{X_{4,1}}(1,1,1,1))  = 6 \cdot 4 - 6 = 18$ by Proposition~4.3. Moreover, $h^0(\mc{O}_{X_{4,2}}(1,1,1,1)) = 6 \cdot 4 = 24$ as $X_{4,2}$ is a disjoint union of its irreducible components.

Consider the finite morphism $$\pi_1: X_{4,1} \uplus X_{4,2} \lra V_4 = X_{4,1} \cup X_{4,2},$$ which by twisting by $\mc{O}_{(\mathbb{P}^2)^{\times 4}}(1,1,1,1)$, we get the exact sequence:

\[
\begin{array}{ll}
0 \lra \mc{O}_{V_4}(1,1,1,1) &\lra [(\pi_1)_{\ast} \mc{O}_{X_{4,1} \uplus X_{4,2}}](1,1,1,1)\\
                                            &\lra \mc{O}_{X_{4,1} \cap X_{4,2}}(1,1,1,1)\\
                                            &\lra 0. 
\end{array}
\]

\vspace{-.6in}

\begin{flushright}
$(\dag)$
\end{flushright}

\vspace{.3in}

\noindent Here, $X_{4,1} \cap X_{4,2}$ is the union of the following irreducible components:

{\small
\[
\begin{array}{lll}
\mathbb{P}_A^1 \times p_a \times p_b  \times p_a, &&p_b \times p_a \times p_b  \times \mathbb{P}_B^1,\\
\mathbb{P}_A^1 \times p_a \times p_c  \times p_a, &&p_c \times p_a \times p_c  \times \mathbb{P}_C^1,\\
\mathbb{P}_B^1 \times p_b \times p_a  \times p_b, &&p_a \times p_b \times p_a  \times \mathbb{P}_A^1,\\
\mathbb{P}_B^1 \times p_b \times p_c  \times p_b, &&p_c \times p_b \times p_c  \times \mathbb{P}_C^1,\\
\mathbb{P}_C^1 \times p_c \times p_a  \times p_c, &&p_a \times p_c \times p_a  \times \mathbb{P}_A^1,\\
\mathbb{P}_C^1 \times p_c \times p_b  \times p_c, &&p_b \times p_c \times p_b  \times \mathbb{P}_B^1,
\end{array}
\]
}

\noindent a union that is not disjoint.
Let $(X_{4,1} \cap X_{4,2})'$ be the disjoint union of these twelve components and consider the finite morphism 
$$\pi_2: (X_{4,1} \cap X_{4,2})' \ra X_{4,1} \cap X_{4,2}.$$ Again by twisting  by $\mc{O}_{\mathbb{P}^2}(1,1,1,1)$, we get the exact sequence:

\[
\begin{array}{ll}
0 \lra \mc{O}_{X_{4,1}\cap X_{4,2}}(1,1,1,1) &\lra [(\pi_2)_{\ast} \mc{O}_{(X_{4,1} \cap X_{4,2})'}](1,1,1,1)\\
                                                                    &\lra \mc{O}_{\mc{S}}(1,1,1,1)\\
                                                                    &\lra 0, 
\end{array}
\]

\vspace{-.6in}

\begin{flushright}
$(\ddag)$
\end{flushright}

\vspace{.3in}

\noindent where $\mc{S}$ is the union of the following six points:

{\small
\[
\begin{array}{ccccc}
p_a \times p_b \times p_a \times p_b,&& p_b \times p_a \times p_b \times p_a,&& p_a \times p_c  \times p_a \times p_c,\\
p_c  \times p_a \times p_c  \times p_a,&& p_b \times p_c \times p_b \times p_c,&& p_c \times p_b \times p_c \times p_b.
\end{array}
\]
}

\medskip

\noindent \underline{Claim 1.} $H^1(\mc{O}_{X_{4,1} \cap X_{4,2}}(1,1,1,1)) = 0$.

\medskip

Note that $H^0([(\pi_2)_{\ast} \mc{O}_{(X_{4,1} \cap X_{4,2})'}](1,1,1,1)) \cong H^0(\mc{O}_{(X_{4,1} \cap X_{4,2})'}(1,1,1,1))$ as $k$-vector spaces since $\pi_2$ is an affine map \cite[Exercise III 4.1]{Hartshorne}. Hence, if Claim 1 holds, then by $(\ddag)$:
\[
\begin{array}{rl}
h^0(\mc{O}_{X_{4,1} \cap X_{4,2}}(1,1,1,1)) &= h^0(\mc{O}_{(X_{4,1} \cap X_{4,2})'}(1,1,1,1)) - h^0(\mc{O}_{\mc{S}}(1,1,1,1))\\
&= 12 \cdot 2 - 6 =18.
\end{array}
\]

\medskip

\noindent \underline{Claim 2.} $H^1(\mc{O}_{V_4}(1,1,1,1)) = 0$.

\medskip

Note that $H^0([(\pi_1)_{\ast} \mc{O}_{X_{4,1} \uplus X_{4,2}}](1,1,1,1)) \cong H^0(\mc{O}_{X_{4,1} \uplus X_{4,2}}(1,1,1,1))$ as $k$-vector spaces since $\pi_1$ is an affine map \cite[Exercise III 4.1]{Hartshorne}. Hence, if Claim~2 is also true, then by $(\dag)$ and the computation above, we note that:

{\small
\[
\begin{array}{rl}
\dim_k B_4 &= h^0(\mc{O}_{V_4}(1,1,1,1))\\
&= h^0(\mc{O}_{X_{4,1} \uplus X_{4,2}}(1,1,1,1)) - h^0(\mc{O}_{X_{4,1} \cap X_{4,2}}(1,1,1,1))\\
&= h^0(\mc{O}_{X_{4,1}}(1,1,1,1)) + h^0(\mc{O}_{X_{4,2}}(1,1,1,1)) - h^0(\mc{O}_{X_{4,1} \cap X_{4,2}}(1,1,1,1))\\
&= 18 + 24 -18 = 24.
\end{array}
\]
}

\noindent Therefore, 
$$\dim_k B_4 = \dim_k S(1,1,1)_4 = 24 \neq 18 = \dim_k P_4.$$

\medskip

Now we prove Claims 1 and 2 above. Here, we refer to the linear components of $(\mathbb{P}^2)^{\times 4}$ of dimensions 1 or 2 by ``lines" or ``planes", respectively.

\medskip

\noindent {\it Proof of Claim 1.} It suffices to show that
$$\theta: H^0\left(\mc{O}_{(X_{4,1} \cap X_{4,2})'}(1,1,1,1)\right) \lra H^0(\mc{O}_{\mc{S}}(1,1,1,1))$$
is surjective. Say $\mc{S}= \{v_i\}_{i=1}^6$, the union of points $v_i$. 
Each point $v_i$ is contained in two lines of $(X_1 \cap X_2)'$, and each of the twelve lines of $(X_1 \cap X_2)'$ contains a unique point of $\mc{S}$.

Choose a basis $\{t_i\}_{i=1}^6$ for $H^0(\mc{S}(1,1,1,1))$, where $t_i(v_j) = \delta_{ij}$. For each $i$, there exists a unique line $L_i$ of $(X_{4,1} \cap X_{4,2})'$ containing $v_i$ so that $pr_{234}(L_i) = pr_{234}(v_i)$. Now we define a preimage of $t_i$ by first extending $t_i$ to a global section $s_i$ of $\mc{O}_{L_i}(1,1,1,1)$. Moreover, extend $s_i$ to a global section $\tilde{s_i}$ on $\mc{O}_{(X_{4,1} \cap X_{4,2})'}(1,1,1,1)$ by declaring that $\tilde{s_i} = s_i$ on $L_i$ and zero elsewhere. 
Now $\theta(\tilde{s_i}) = t_i$ for all $i$, and $\theta$ is surjective. \qed

\medskip

\noindent {\it Proof of Claim 2.}
It suffices to show that 
$$\tau: H^0(\mc{O}_{X_{4,1} \uplus X_{4,2}}(1,1,1,1)) \lra H^0(\mc{O}_{X_{4,1} \cap X_{4,2}}(1,1,1,1))$$
is surjective. 

Recall that $X_{4,1} \cap X_{4,2}$ is the union of twelve lines $\{L_i\}$, and $X_{4,1} \uplus X_{4,2}$ is the union of twelve planes $\{P_i\}$. Here, each line $L_i$ of $X_{4,1} \cap X_{4,2}$ is contained in precisely two planes of $X_{4,1} \uplus X_{4,2}$, and each plane $P_i$ of $X_{4,1} \uplus X_{4,2}$ contains precisely two lines of $X_{4,1} \cap X_{4,2}$.

Choose a basis $\{t_i\}_{i=1}^{12}$ of $H^0\left(\mc{O}_{X_{4,1} \cap X_{4,2}} (1,1,1,1)\right)$ so that $t_i(L_j) = \delta_{ij}$. For each $i$, we want a preimage of $t_i$ in $H^0\left(\mc{O}_{X_{4,1} \uplus X_{4,2}} (1,1,1,1)\right)$.

Say $P_i$ is a plane of $X_{4,1} \uplus X_{4,2}$ that contains $L_i$, and $L_j$ is the other line that is contained in $P_i$. Since $\mc{O}_{P_i}(1,1,1,1)$ is very ample, its global sections separate the lines $L_i$ and $L_j$. In other words, there exists $s_i \in H^0(\mc{O}_{P_i}(1,1,1,1))$ so that $s_i(L_k)= \delta_{ik}$. 
Extend $s_i$ to $\tilde{s_i} \in H^0\left(\mc{O}_{X_{4,1} \cap X_{4,2}} (1,1,1,1)\right)$ by declaring that $\tilde{s_i} = s_i$ on $L_i$, and zero elsewhere.
Now $\tau(\tilde{s_i}) = t_i$ for all $i$, and $\tau$ is surjective.  \qed

\medskip

\noindent {\bf Acknowledgments.} I thank Sue Sierra for pointing out a typographical error in Lemma 3.9, and for providing several insightful suggestions. I also thank Paul Smith for suggesting that a quiver could be used for the bookkeeping required in Proposition \ref{prop:Vd}. Moreover, I am grateful to Paul Smith and Toby Stafford for providing detailed remarks, which improved the exposition of this manuscript.

\bibliography{biblio}

\end{document}